\documentclass[12pt,a4paper]{amsart}
\usepackage[left=1.4in,right=1.4in,top=1.5in,bottom=1.5in]{geometry}
\usepackage[utf8]{inputenc}
\usepackage[T1]{fontenc}
\usepackage{amsmath,bbm,amsthm,amssymb,amsfonts,mathtools}
\newcommand{\set}[1]{\ensuremath{\left\{#1\right\}}}
\newcommand{\setp}[1]{\ensuremath{\left(#1\right)}}

\newcommand{\EN}{\ensuremath{\set{1,2,\dotsc,n}}}
\newcommand{\undX}{\underline{X}}
\newtheorem{lemma}{Lemma}
\usepackage{algorithmicx}
\usepackage{algpseudocode}
\usepackage{graphicx}
\usepackage{tikz}
\usetikzlibrary{calc,backgrounds,fit}

\newcommand{\qtilde}{{\tilde q}}
\newcommand{\sumn}{\sum_{k=1}^n}
\newcommand{\probaK}{\binom{n}{k} (1-x)^{k} x^{n-k}}
\newcommand{\upc}[1]{{\MakeUppercase #1}}

\title{Weber's optimal stopping problem and generalizations}
\author{Rémi Dendievel\\Université Libre de Bruxelles}
\address{Université Libre de Bruxelles\\
Campus Plaine, CP 210\\
B-1050 Bruxelles\\
Belgique}

\begin{document}



\maketitle

\begin{abstract}
  One way to interpret the classical secretary problem (CSP) is to
consider it as a special case of the following problem. We observe $n$
independent indicator variables $I_1,I_2,\dotsc,I_n$ sequentially and
we try to stop on the last variable being equal to 1. If $I_k=1$ it
means that the $k$-th observed secretary has smaller rank than all
previous ones (and therefore is a better secretary). In the CSP
$p_k=E(I_k)=1/k$ and the last $k$ with $I_k=1$ stands for the best
candidate. The more general problem of stopping on a last ``1'' was
studied by Bruss~(2000). In what we will call Weber's problem the
variables $I_k$ can take more than two values and we try to stop on
the last occurence of \textit{one} of these values. Notice that we do
not know in advance the value taken by the variable on which we stop.

We can solve this problem in some cases and provide algorithms to
compute the optimal stopping rule. These cases carry enough generality
to be applicable in concrete situations.



\end{abstract}

\section{Statement of the type of Problems}
\label{sec:formalization}

The following problem has been proposed in 2013  by Weber (R.\,R.~Weber,
University of Cambridge) to his students.

\subsection{Problem 1 (Weber's problem)}
\label{sec:problem-1-webers}
\begin{quote}\em
  A financial advisor tries to impress his clients if immediately
  following a week in which the \textsc{ftse} index moves by more than
  5\% in some direction he correctly predicts that this is the last
  week during the calendar year that it moves more than 5\% in
  \emph{that} direction.

  Suppose that in each week the change in the index is independently
  up by at least 5\%, down by at least 5\% or neither of these, with
  probabilities $p$, $p$ and $1-2p$ respectively ($p\le 1/2$). He
  makes at most one prediction this year. With what strategy does he
  maximize the probability of impressing his clients?
\end{quote}

The solution by backwards induction is relatively straightforward and
several students of Weber's found the solution for this specific
problem. Weber then discussed with Bruss (private communication
2013) several modifications of this problem. The objective of the
present paper is to present the solutions of two Bruss-Weber
modifications which have an appeal for applications.

We present two modifications of Problem~1. We code a ``1'' if the
index goes up by some fixed percentage some day, ``$-1$'' if it goes
down by some (other) fixed percentage and ``0'' else. Two other ways
of generalization come to our mind. One can imagine that the
probabilities of a ``$1$'' and of a ``$-1$'' are different, we then
have two parameters $p$ and $p'$.  One can also imagine that the
probabilities of a ``1'' or ``$-1$'' are equal but are allowed to
differ day after day. We then have $n$ parameters $(p_k,k=1,\dots,n)$,
thus also generalizing Bruss' problem of stopping on a alast specific
success. This also opens the way to tackle continuous time problems
with a random number of decision items (see Bruss (2000)). In this
paper we confine our interest to the discrete time setting, however.

The integer $n$ is always known and represents the number of
observations of the index made over the time horizon. In Weber's
problem, the model is as follows.  We call $n$ the length of the
horizon. Let $X_1,X_2,\dotsc,X_n$
($n$ known) be i.i.d. random variables, such that
\begin{displaymath}
  P(X_i=1)=P(X_i=-1)=p \quad\text{and}\quad P(X_i=0)=1-2p.
\end{displaymath}
Let $(\mathcal{F}_k)_{k=1,2,\dotsc}$ be the filtration generated by
$X_1,X_2,\dotsc,X_k$. We want to find, among all random times $\tau$,
the stopping time $\tau^\star$ maximizing the quantity
\begin{equation}\label{eq:4}
  E\left(
    \mathbbm{1}\biggl[ 
    \biggl( 
    \set{X_\tau=1}\cap\bigcap_{i=\tau+1}^n\set{X_i\neq
      1}
    \biggr)
    \cup
    \biggl( 
    \set{X_\tau=-1}\cap\bigcap_{i=\tau+1}^n\set{X_i\neq -1}
    \biggr)
    \biggr] 
  \right).
\end{equation}

\noindent\textbf{Remark.} Note that stopping on a ``$1$'' or a
``$-1$'' may be a stopping time but stopping on a last ``1'' or ``-1''
is in general not a stopping time. Adding the conditional knowledge
$\mathcal F_k$ in the expectation we obtain a stopping time. However
this can be dropped because of the markovian nature of the
problem. Indeed, all that counts at a given time $k$ is the value of
$I_k$ and the number of remaining variables. The knowledge of the
history therefore does not influence our decision at any given time.

We now describe quickly the two modifications of Problem~1.

\subsection{Problem~2}
\label{sec:problem-2}

The difference between this problem and Problem~1 is that here
the probability of a ``$1$'' is different from that of a ``$-1$''.
This problem is therefore described by three parameters: $n$, $p$ and
$p'$ which are the number of variables, the probability of a
variable equal to $1$ and the probability a variable equal to $-1$,
respectively.

\subsection{Problem~3}
\label{sec:problem-3}

Another interesting modification is to look at the problem of a ``same
$p$'' for $+1$'s and $-1$'s but one that is changing over
time. Formally, the problem is described by the parameters $n$ and
$p_1,p_2,\dotsc,p_n$ where $p_k=P(X_k=1)=P(X_k=-1)$.

We state them and we provide optimal decision rules. Our special
interest is to make these solutions as quick and concise as possible.

\section{Solution of Problem 2}
\label{sec:problem-2}

We are looking at the case where the probabilities of a ``1'' or a
``$-1$'' are not necessarily equal. The variables $X_1,X_2,\dotsc,X_n$
are i.i.d with $p:=P(X_1=1)$, $p':=P(X_1=-1)$ and $P(X_1=0)=1-p-p'$,
$p,p'\in[0,1]$. Without loss of generality we will suppose that $p\ge
p'$. We must suppose that $p+p'\le 1$. If $p+p'=1$ the problem is
trivial since it suffices to stop on the last event. If $p'=0$, Bruss'
odds-theorem and the accompanying algorithm gives the optimal
strategy. Therefore we will suppose that $p,p'\in (0,1)$ and, to avoid
trivial cases, that $p+p'<1$. The notations $q=1-p$, $q'=1-p'$ and
$\qtilde = 1-p-p'$ will be used thoughout the article.

\subsection{Monotonicity and unimodality}
\label{sec:monotonicity}

In this section we state and prove several lemmas that will be the
basis of the solving algorithm.

We will show that the problem is monotone, that is, if at a certain
time index it is optimal to stop on a ``1'' (respectively on a
``$-1$''), then it is optimal to stop on a ``1'' (respectively on a
``$-1$'') at any later time index. Assaf and Samuel-Cahn (2000) called
such stopping rules ``simple''.

\begin{lemma}
  \label{sec:monotonicity-1}
  The problem described in Problem~2 is monotone.
\end{lemma}
\begin{proof}
  Let $v_j$ be the optimal probability of a win when there are still
  $j$ variables to observe. It is optimal to stop at stage
  $k\in\set{1,2,\dots,n}$ if $X_k=1$ and
  \begin{equation}
    \label{eq:1}
    q^{n-k}>v_{n-k},
  \end{equation}
  by definition of $v_{n-k}$. We now show that \eqref{eq:1} implies
  $q^{n-k-1}>v_{n-k-1}$.  From independence and the optimality
  principle we have \[v_{n-k} = p(v_{n-k-1}\vee q^{n-k-1}) +
  p'(v_{n-k-1}\vee (q')^{n-k-1}) + (1-p-p') v_{n-k-1},\] where $a\vee
  b$ denotes the maximum of $a$ and $b$. Also,
  $\max\{v_{n-k-1},q^{n-k-1},(q')^{n-k-1}\} \ge v_{n-k-1}$, and thus
  by inequality~(\ref{eq:1}) we obtain
  \begin{align*}
    \label{eq:3}
    q^{n-k-1} &> \bigl(p v_{n-k-1} + p' v_{n-k-1} + (1-p-p') v_{n-k-1}\bigr) q^{-1} \\
    &> v_{n-k-1}, \qquad 0 \le k \le n-1.
  \end{align*}
  This shows monotonicity with respect to the stopping problem of the
  ``1'''s. An analoguous argument proves monotonicity for the value
  $-1$. Hence the lemma is proved.
\end{proof}

Since $p\ge p'$, we expect a different behaviour regarding the 1's and
the $-1$'s, and this will become apparent in what follows. For
$j,k\in\EN$, we use the following notations
\begin{align*}
  \tau^{+}_j &= \inf\set{j\le m\le n : X_m=1} \wedge n, \\
  \tau^{-}_k &= \inf\set{k\le m\le n : X_m=-1} \wedge n,\\
  \tau_{j,k} &= \tau^{+}_j \wedge \tau^{-}_k,
\end{align*}
with $x\wedge y=\min\set{x,y}$ and the usual convention that $\inf
\varnothing = +\infty$.  The monotonicity of the problem implies that
there are two indexes $s$ and $s'$ such that the stopping time
$\tau_{s,s'}$ is optimal, that is, maximizes~(\ref{eq:4}).

Consider a stopping region (or stopping set) as follows
\begin{equation}
  \label{eq:6}
  R = \set{(k,1) : k\in J^+} \cup \set{(k,-1) : k\in J^-} \cup \set{(n,0)}
\end{equation}
where $J^+$ and $J^-$ are subsets of $\EN$. A stopping time for this
problem can be defined as the first $k\in \EN$ such that $(k, X_k)\in
R$, that is, the first hitting time of the set $R$ of the process
$((k, X_k))_{1\le k\le n}$. This hitting time will be denoted
$\tau_{R}$. For any stopping region $R$, we let $V_{R}$ be the
probability of winning if our strategy is to we use the first hitting
time of the set $R$.

Let $R^\star$ denote the optimal stopping region, that is,
\begin{equation}
  \label{eq:7}
  R^\star = \set{(k,1) : k\ge s} \cup \set{(k,-1) : k\ge s'}\cup\set{(n,0)}.
\end{equation}
By definition, the optimal value of the problem is
$V:=V_{R^\star}$. We now have two lemmas:

\begin{lemma}
  If $R_1 \subseteq R_2 \subseteq R^\star$, then $V_{R_1}\le V_{R_2}$.
\end{lemma}
\begin{proof} If $R_1=R_2$, there is nothing to prove. Let
  $R_1\neq R_2$ and $\underline X = (X_1,X_2,\dotsc,X_n)$ and suppose
  that $\tau_{R_2}(\undX)\neq \tau_{R_1}(\undX)$. We necessarily have
  $\tau_{R_2}(\undX)< \tau_{R_1}(\undX)$ because $R_1\subseteq R_2$,
  and it is impossible to hit $R_1$ before $R_2$. We also have $s\le
  \tau_{R^\star}(\undX) \le \tau_{R_2}(\undX)$. Let
  $t=\tau_{R_2}(\undX)$. We assume that $X_t=1$. Let $\underline
  X'=(0,0,\dotsc,0,1,X_{t+1}, \dotsc,X_n)$. Since $t$ is the first
  index $k$ after $s$ such that $X'_k=1$ it is optimal to stop
  $\underline X'$ at time $t$. We use the same argument if we stop on
  a $-1$. Therefore, the strategy based on $R_2$ does better than the
  strategy based on $R_1$.

  On the set $\set{\tau_{R_1} = \tau_{R_2}}$, the two strategies
  perform the same. Overall, the strategy based on $R_2$ is better.
\end{proof}

\begin{lemma}
  If $R^\star \subseteq R_1 \subseteq R_2$, then $V_{R_1} \ge V_{R_2}$.
\end{lemma}
\begin{proof}
  The arguments for this proof are analogous to those of the proof of
  the previous lemma and can be omitted.
\end{proof}

Finally, the following lemma allows, as we will see, for an efficient
algorithm.
\begin{lemma}
  If $p\ge p'$, then $s \ge s'$.
\end{lemma}
\begin{proof}
  We must show that stopping on $X_s=-1$ is the optimal choice. In
  fact, we have, with the same notation as above,
  \begin{displaymath}
    {(q')}^{n-s} \ge q^{n-s} > v_{n-s},
  \end{displaymath}
  and hence in particular $s\ge s'$, since $q'>q$.
\end{proof}

\subsection{Recurrence equations for  $w_{j,k}$}
\label{sec:few-recurr-relat}

In this section we write the recurrence equations of the success
probabilities of the stopping times $\tau_{j,k}$ defined in
Section~\ref{sec:monotonicity}. Let $w_{j,k}=w^{(n)}_{j,k}$ denote the
probability of a win if we use the stopping time $\tau_{j,k}$. If $j >
k$, by conditionning on the value of $X_k$,
\begin{align}
  &\begin{aligned}
    \label{eq:rec1}
    w_{j,k} &= \sum_{\nu = -1, 1, 0} P(X_j=\nu) P(\text{win using
      $\tau^+_{j}\wedge\tau^-_j$}|X_j=\nu)\\
    &= p' {(q')}^{n-k} + p w_{j,k+1} + (1-p-p') w_{j,k+1}\\
    &= p' {(q')}^{n-k} + q' w_{j,k+1}.
  \end{aligned}
\end{align}
If $j=k$, we have
\begin{equation}
  \label{eq:rec2}
  w_{j,j} = p'{(q')}^{n-j} + pq^{n-j} + \qtilde w_{j+1,j+1}.
\end{equation}
where $\qtilde = 1 - p - p'$. We know that $w_{n,n}=p+p'$ because the
probability of a win using the strategy ``stop on the last variable''
is simply the probability that $X_n$ is not equal to zero.

Writing out these recurrence equations leads us to the following formulas:
\begin{equation}
  \label{eq:w-expl1}
  w_{j,j} = \frac{p}{p'} (q^{n-j+1} - \qtilde^{n-j+1}) + \frac{p'}{p}
  ({(q')}^{n-j+1} - \qtilde^{n-j+1}),
\end{equation}
and for $j > k$, we have
\begin{equation}
  \label{eq:w-expl2}
    w_{j,k} =
  (j-k) p' {(q')}^{n-k} + {(q')}^{j-k} \biggl(\frac{p}{p'} (q^{n-j+1} - {\tilde q}^{n-j+1})
  + \frac{p'}{p} ({(q')}^{n-j+1} - {\tilde q}^{n-j+1})\biggr).
\end{equation}
If $j < k$ we use the same expression and exchange the role of $j$ and
$k$.

\subsection{Graphical illustrations}
\label{sec:graph-ill}

In the example displayed in figure~\ref{fig:plot-40}, we choose
$n=40$, $p=0.09$ and $p'=0.05$. We find that the optimal thresholds
are $(s,s')=(33, 28)$ and the optimal value is
$v=w^{(40)}_{33,28}=0.52987\dots$. We plotted the $w_{j,k}$'s for this
choice of $n$, $p$ and $p'$. We notice that the maximum is obtained
for $k_+=s$, $k_-=s'$.

\begin{figure}[h]
  \begin{center}
    \includegraphics[width=0.9\linewidth]{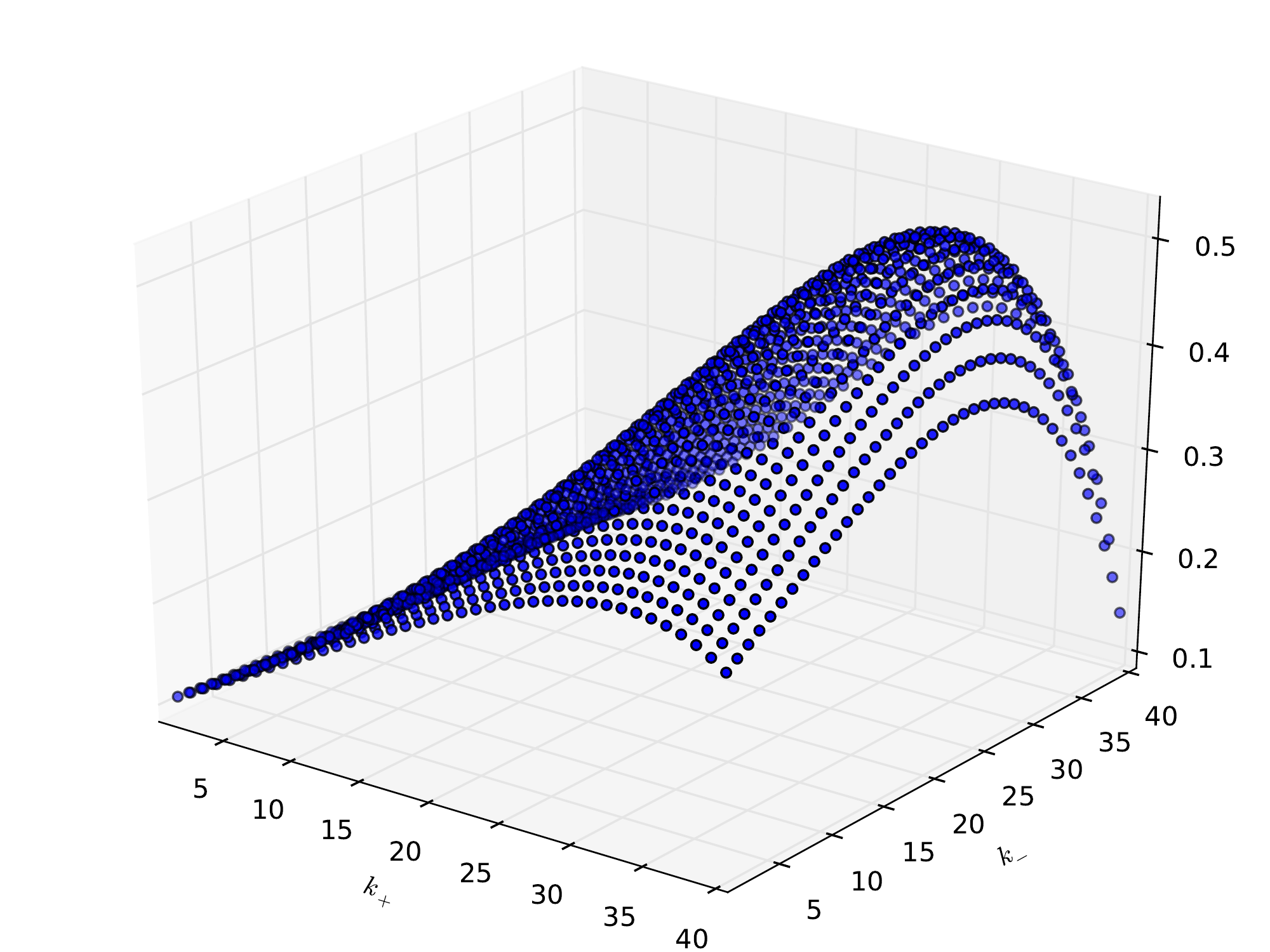}
  \end{center}
  \caption{The plot shows the value $w_{k_+,k_-}$, for a
    Problem~2-framework with $n=40$, $p=0.09$, $p'=0.05$.}
\label{fig:plot-40}
\end{figure}

\subsection{Solving algorithm}
\label{sec:recurs-based-algor}

We present a first algorithm (see figure~\ref{fig:algo-model-1}) based
on the properties shown in Section~\ref{sec:monotonicity}. The idea of
the algorithm is the following: if we start with the stopping region
$R_0 = \set{(n,1), (n,0), (n,-1)}$ and if we carefully add points to
$R_0$, we will be able to detect the indexes $s$ and $s'$, using the
unimodality property.

The recurrence equations \eqref{eq:rec1} and \eqref{eq:rec2} can be
used to speed up the computations of the $w_{j,k}$'s on lines 6, 12
and 23 in the algorithm.

\noindent%
\begin{figure}[h]
\begin{center}
\fbox{\begin{minipage}{0.9\linewidth}
  \begin{algorithmic}[1]
    \State Input : ($n,p,p'$)
    \State $k \gets n$.
    \State $\alpha \gets w_{k,k}$
    \State Label1
    \If{$k=1$} $s:=1$, $s':=1$.
    \EndIf
    \State $\beta \gets w_{k,k-1}$
    \If{$\beta<\alpha$}
    \State $s:=k$, $s':=k$, and stop.
    \Else
    \State $\alpha \gets \beta$
    \State $k \gets k-1$
    \State $\beta \gets w_{k,k}$
    \If{$\beta < \alpha$}
    \State $s:=k$ and go to Label2
    \Else
    \State $\alpha \gets \beta$
    \State go to Label1
    \EndIf
    \EndIf
    \State Label2
    \If{$k=1$} $s':=1$.
    \EndIf
    \State $\beta \gets w_{s,k-1}$
    \If{$\beta<\alpha$}
    \State $s':=k$ and stop
    \Else
    \State $\alpha \gets \beta$
    \State $k \gets k-1$
    \State go to Label2
    \EndIf
  \end{algorithmic}
\end{minipage}}
\end{center}
\caption{Pseudo-code for the algorithm giving the two
  optimal thresholds $s$ and $s'$, in model~1. The input of the
  algorithm is $(n,p,p')$.}
\label{fig:algo-model-1}
\end{figure}

\subsection{Optimality of the computed thresholds}
\label{sec:optim-comp-thresh}

The lemmas of the above section do not guarantee that the indexes
found in the algorithm are the true $s$ (line~13) and $s'$. We must
verify that if $s'<s$, then $w_{s-1,s-1} \le w_{s,s-1}$. This is true,
because the two strategies $\tau_{s-1,s-1}$ and $\tau_{s,s-1}$ behave
the same except in the case where $X_{s-1}=1$.  Therefore, to compare
$w_{s,s-1}$ against $w_{s-1,s-1}$, one can look at the same optimal
stopping problem but with a number of variables being equal to
$n-s$. In this case we see that the optimal strategy has success
probability $w_{2,1}\ge w_{1,1}$ which is equal to the value
$w_{s,s-1}\ge w_{s-1,s-1}$ of the initial problem with $n$ variables.


\subsection{Formula-based algorithm}
\label{sec:remark}

Here we use the explicit formula of the $w_{j,k}$'s given by
\eqref{eq:w-expl2}. Note that this equation remains correct even if
$j=k$. We also know that $s\ge s'$ because $p>p'$. As a result we know
that we will not need to look at the value of $w_{j,k}$ for
$j<k$. Therefore this is the only formula we need.

Another way of finding $s$ is by putting the second
``$(-1)$-threshold'' to 1, and comparing $w_{n,1}$ with $w_{n-1,1}$,
then $w_{n-1,1}$ with $w_{n-2,1}$, etc. By an argument to the one used
in Section~\ref{sec:optim-comp-thresh}, we can show that this
determines correctly $s$.

We will use this and the unimodality of the function $j\mapsto
w_{j,1}$ to find the index $s$. With a bisection algorithm this can be
done in $O(\ln (n))$ time. When this first index is found, we look for
the $k$ maximizing $k\mapsto w_{s,k}$.

An important feature of this method is that, for any $k \le s'$, the
mode of the function $j\mapsto w_{j,k}$ occurs at index $s$.

%
%

In figure~\ref{modes}, we show information about the shape of the
graph of the function $(j,k)\mapsto w_{j,k}$. Intuitively, if $k_-=n$,
it is not likely that we win by selecting a ``$-1$''. Therefore, we
must concentrate on selecting a ``$1$'' only. The problem then becomes
close to problem that the odds-algorithm can solve optimally.

\begin{figure}
  \begin{center}
    \includegraphics[width=0.9\linewidth]{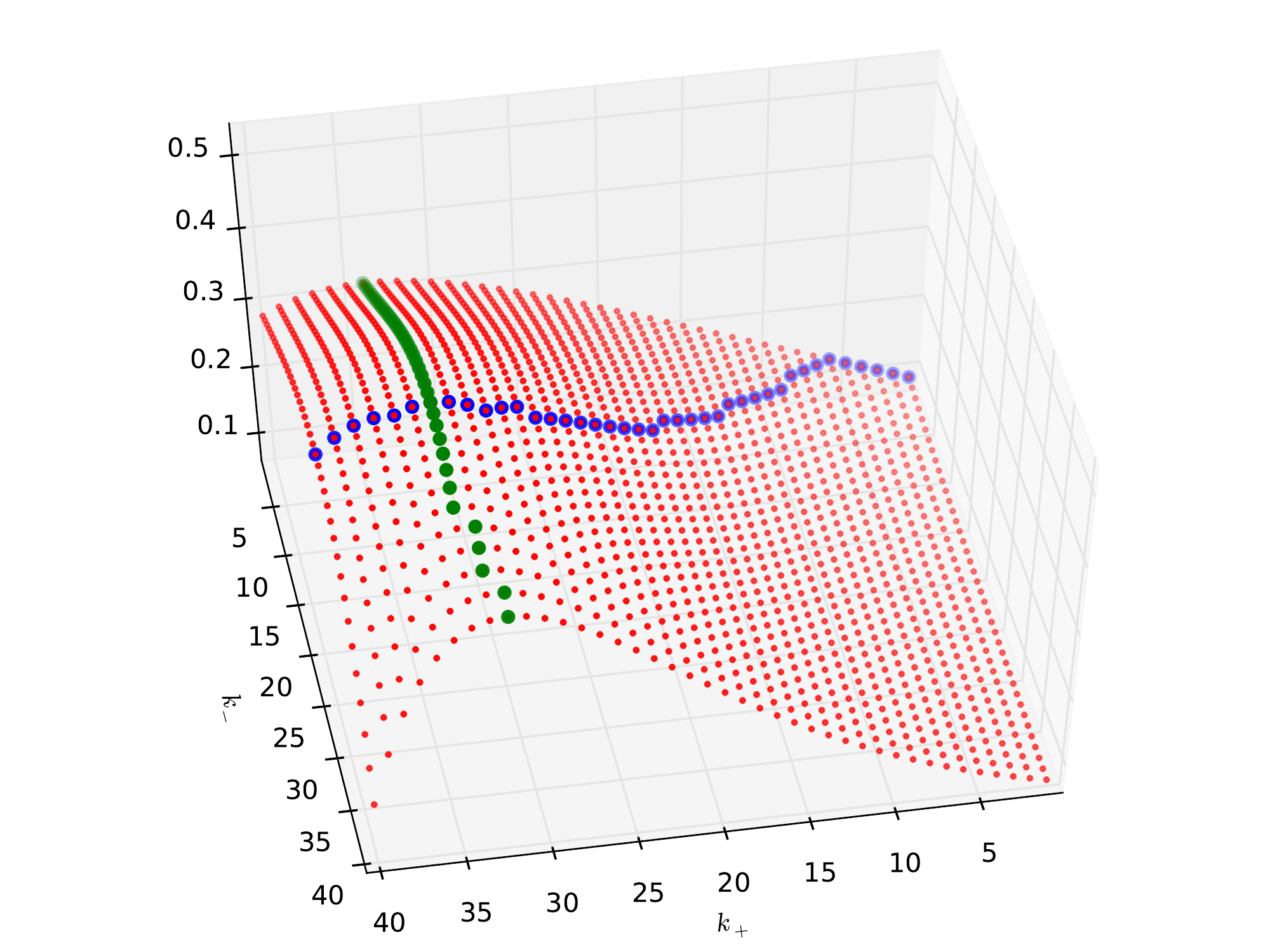}
  \end{center}
  \caption{In this example (problem~2) we chose
    $(n,p,p')=(40,0.1,0.05)$. The green dots show the modes, for
    $k_-=1,2,\dotsc,n$, of the functions $j\mapsto w_{j,k_-}$. The
    blue dots show the modes, for $k_+=1,2,\dotsc,n$, of the functions
    $k\mapsto w_{k_+,k}$.}
  \label{modes}
\end{figure}


\section{Solution of Problem 3}
\label{sec:problem-3}

In this section we suppose that the variables $X_1,X_2,\dots,X_n$ are
independent but have different distributions, that is, there are known
parameters $p_1,p_2,\dotsc,p_n$ such that for all $1\le i \le n$,
\begin{displaymath}
  P(X_i=1)=P(X_i=-1)=p_i \quad\text{and}\quad P(X_i=0)=1-2p_i.
\end{displaymath}
We put as usual $q_j=1-p_j$.

\subsection{Monotonicity}
\label{sec:monotonicity-2}

The proof of monotonicity of Section~\ref{sec:problem-2} can be
adapted and using parameters $p_k$'s instead of one single parameter
$p$ does not lengthen the proof.

By the same result about monotone stpping problems used in
Section~\ref{sec:monotonicity-1}, we know that the 1-stage look-ahead
rule is optimal (see Ferguson, 2008). Define
\begin{align*}
  V_k &= P\setp{\text{we win by selecting the opportunity at stage
      $k$}|X_k\in\set{-1,1}} \\
\intertext{and}
  W_k &= w_{k+1,k+1}
\end{align*}
where the $w_{j,k}$ have the same meaning as in the previous model.
Obviously $V_k=\prod_{j=k+1}^n q_j$, and
\begin{align*}
  W_j &= \sum_{m=0}^{n-k-1}
    \left[ \prod_{j=k+1}^{k+m}(1-2p_j)\right] (2p_{k+m+1})
    \left[ \prod_{j=k+m+2}^{n}(1-p_j)\right] \\
  &= 2 \sum_{m=0}^{n-k-1}
    \left[ \prod_{j=k+1}^{k+m}(q_j-p_j)\right] p_{k+m+1}
    \left[ \prod_{j=k+m+2}^{n}q_j\right].
\end{align*}
Therefore,
\begin{equation}
  \frac{W_k}{V_k} = 2 \sum_{m=0}^{n-k-1}
  \left[ \prod_{j=k+1}^{k+m} (1-r_j) r_{k+m+1}\right],
\end{equation}
where $r_j=p_j/q_j$. It is not easy to extract the index $k^\star$
such that $W_k/V_k \ge 1$ for $k<k^\star$ and $W_k/V_k < 1$ for $k\ge
k^\star$. Nevertheless, we can provide an algorithm with linear
complexity that will compute $k^\star$.

\subsection{Computation of the stopping threshold}
\label{sec:stopping-index}

We must start by looking at the end of the sequence. If
$W_{n-1}/V_{n-1}\ge 1$ we must stop on the last variable. If the ratio
is smaller, then look at the ratio at time $n-2$, and so on. By
starting with the end of the sequence we will only compute the values
used in the expression of $W_{k^\star}/V_{k^\star}$.

Let $\Lambda_k=W_k/(2V_k)$, $k=1,2,\dotsc,n-1$. We see that
$\Lambda_{n-1}=r_n$. And we have the following recurrence equation
\begin{equation}
  \Lambda_k = r_{k+1} + (1-r_{k+1}) \Lambda_{k+1}, \qquad 1\le k\le n-2.
\end{equation}
Here is, in pseudocode, the algorithm that computes the time index
$k^\star$.
\medskip
\begin{list}{--}{}
\item \textbf{Step 0} Set $k=n-1$. Compute $r_n$. We have
  $\Lambda_{n-1}=r_n$.
\item \textbf{Step 1} $k\leftarrow k-1$. Compute $r_{k+1}$, $1-r_{k+1}$ and
  $\Lambda_k = r_{k+1} + (1-r_{k+1})\Lambda_{k+1}$.
\item \textbf{Step 2} If $2\Lambda_k \ge 1$ then \textbf{stop}, and
  $k^\star=k+1$. If $2\Lambda_k<1$, go to step 1.
\end{list}\medskip

\noindent If $\Lambda_1 < 1$, set $k^\star=1$.

We see that the complexity of the algorithm is linear, and that the
stopping time (1-sla) defined by
\begin{equation}
  \label{eq:5}
  n \wedge \inf\set{k>k^\star: X_k\in\set{-1,1}}
\end{equation}
is optimal.

\section{Particular case}
\label{sec:particular-case}

If $p=p'$ in model~1 or $p_k=p_1=:p$ for all $k$ in Problem~3, then we
have the Weber's original problem. In this case we use the method
described in model~2 to obtain the optimal strategy. Since
$V_k=q^{n-k}$ and $W_k=2((1-p)^{n-k} - (1-2p)^{n-k})$ we can describe
the optimal stopping rule as: stop on the first variable equal to $1$
or equal to $-1$ after time $s$ ($s$ included), where $s$ is the
smallest integer of $\EN$ such that $W_k / V_k<1$. If there is no such
integer, put $s=1$.

This case is also obtained when we look at Problem~2 with $p=p'$.

\section*{One more generalization}
\label{sec:one-more-gener}

The two models solved in Section~\ref{sec:monotonicity} and
Section~\ref{sec:problem-3} have a common generalization. This model
deals with variables $X_1,X_2,\dots,X_n$, for a fixed $n\ge 1$, who
are independent, and where $P(X_k=1)\neq P(X_k=-1)$.

We believe that the approach described by the solving algorithm of
problem~2 can be used to solve this new generalization. The algorithm
would require three test moves instead of two. And since the
assumptions of problem~2 seem to be general enough to be useful in
applications, we will leave out these new technicalities.

%
%


\section{Continuous-time Approximation}
\label{sec:cont-time-appr}

Approximating the model cannot give us optimal answers. But using this
approach we are able to give lower bounds for the optimal success
probability of the real discrete time model (problem~1). The model is
as follows.

Let $X_1,X_2,\dotsc,X_n$ be the $n$ random variables of Weber's
original problem, where $P(X_i=1)=P(X_i=-1)=p<1/2$. Now let $U_i$,
$i=1,2,\dotsc,n$ be $n$ independent random variables uniformly
distribution over the interval $[0,1]$. Then define $T_i=U_{(i)}$,
that is, $T_i$ is the $i$th order statistic of the $U_i$'s. We look at
$T_i$ as the arrival time of variable $X_i$.

Suppose that now we are not able to count time discretely and observe
one variable every second, but instead we are able to scan the time
interval $[0,1]$ from left to right and we detect a variable at each
of the arrival times $T_i$. The number $n$ of variables is known.

We restrict our choice of strategies to fixed time threshold
strategies: before observing the variables, we choose an $x\in[0, 1]$
after which we decide to select the first non zero variable
appearing. Following Bruss (2000) we call such strategies
$x$-strategies. We will now write the expression of $p_n(x)$, the
success probability of succeeding in finding the last non zero
variable of its kind (as in Weber's problem), now in terms of an
$x$-strategy. Then we maximize this probability over all $x\in[0, 1]$.

Let $N_x$ be the exact number of variables arriving after time $x$. It
is a random variable, let $N_x^-$ the exact number of $-1$'s arriving
after $x$. Note that if we select a ``$+1$'', then if there are also
``$-1$'' variables after time $x$, we are certain that they come after
this ``$+1$'' variable as, according to the $x$-strategy, we want to
select the first non zero variable after time $x$.

In what follows, the probabilities are all taken
under the condition $N_0=n$. We have
\begin{align*}
  p_n(x) &=
  \sumn P(N_x = k) P(\text{success of the $x$-strategy} | N_x =
  k)
  \\ &=
  \sumn \probaK \left( \sum_{m=0}^{k-1}
    2P(\text{success, we select a ``$+1$'', $N_x^-=m$}|N_x=k)\right) \\ &=
  2\sumn \probaK \left( \sum_{m=0}^{k-1} \binom{k}{m+1} p^{m+1}
    \qtilde^{k-(m+1)} \right) \\ &=
  2\sumn \probaK \left( \sum_{m=1}^{k} \binom{k}{m} p^{m} \qtilde^{k-m}
  \right)
\end{align*}
and we use Newton's formula several times to
finish this computation:
\begin{align*}
  p_n(x)&=
  2\sumn \probaK \left( (p+\qtilde)^k - \qtilde^k \right) \\ &=
  2\sumn \probaK \left( q^k - \qtilde^k \right) \\ &=
  2\sumn \binom nk ((1-x)q)^k x^{n-k} - \sumn \binom nk
  ((1-x)\qtilde)^k x^{n-k} \\ &=
  2 ((q+px)^n - (\qtilde + 2px)^n)
\end{align*}
Maximizing the function $p_n(x)$, we obtain
\begin{equation*}
  \frac{\mathrm dp_n}{\mathrm dx}(x^\star_n) = 0 \iff x^\star_n =
  \frac 1p \frac{q-\qtilde \beta_n}{2\beta_n - 1}
\end{equation*}
where $\beta_n = \sqrt[n-1]{2}$. The value of this maximum is
\begin{align*}
  p_n(x^\star_n) &=
  2\left(
    \left( q + \frac{q-\qtilde\beta_n}{2\beta_n-1} \right)^n -
    \left( \qtilde + 2\frac{q-\qtilde\beta_n}{2\beta_n-1} \right)^n
  \right),\\
  \intertext{as we can see by the expression of $p_n(x)$.
    Using $q=\qtilde + p$ and $q+p=1$, this can be written }
  p_n(x^\star_n)&=
  2\left(
    \left(q + \frac{q(1-2\beta_n) + \beta_n}{2\beta_n-1} \right)^n -
    \left(\qtilde + \frac{\qtilde(1-2\beta_n) + 1}{2\beta_n-1} \right)^n
  \right) \\ &=
  2 \left(
    \left(\frac{\beta_n}{2\beta_n-1}\right)^n -
    \left(\frac{1}{2\beta_n-1}\right)^n
  \right) \\ &=
  2 \frac{\beta_n^{\,\raisebox{2pt}{\scriptsize$n$}} - 1}{(2\beta_n -
    1)^n} \\ &=
  2(2\beta_n - 1)^{1-n}
\end{align*}
We can show that, $p_n(x^\star_n)\searrow \frac12$ as $n\to +\infty$.

We notice interesting features. First, as intuition might tell, for a
fixed $n$, if $p$ is too small, then the success probability goes down
and can be arbitrarily small ($\lim_{p\to0} p_n(0)=0$). The second
observation is more interesting. The parameter $p$ does not appear in
the expression of the optimal success probability for a fixed $n$ and
$p$.

The conclusion of this continuous approximation is that, when $p\ge
\frac{\beta_n-1}{2\beta_n-1}$, it is always possible to achieve a
success probability that is at least equal to $1/2$. And $1/2$ is also
a lower bound for the optimal success probability of Problem~1.


\nocite{*}

\bibliographystyle{plain}

\begin{thebibliography}{1}

\bibitem{assaf00:_simpl}
D.~Assaf and E.~Samuel-Cahn.
\newblock Simple ratio prophet inequalities for a mortal with multiple choices.
\newblock {\em Journal of Applied Probability}, 37(4):1084--1091, 2000.

\bibitem{bruss03:_bound_odds_algor}
F.T. Bruss.
\newblock A note on bounds for the odds algorithm of optimal stopping.
\newblock {\em Annals of Probability}, 31(4):1859--1862, 2003.

\bibitem{bruss00:_sum}
F.~T. Bruss.
\newblock Sum the odds to one and stop.
\newblock {\em Annals of Probability}, 28(3):1384--1391, 2000.


\bibitem{hsiau02:_selec_last_succes_markov_depen_trial}
S.-R. Hsiau and J.-R. Yang.
\newblock Selecting the last success in \upc markov-dependent trials.
\newblock {\em Journal of Applied Probability}, 93(2):271--281, 2002.


\bibitem{matsui2012lower}
T.~Matsui and K.~Ano.
\newblock Lower bounds for \upc bruss' odds problem with multiple stoppings.
\newblock {\em Arxiv preprint arXiv:1204.5537}, 2012.


\bibitem{tamakimaximizing}
M.~Tamaki.
\newblock Maximizing the probability of stopping on any of the last $m$
  successes when the number of observations is random.
\newblock {\em Advances in Applied Probability}, 43(3):760--781, 2011.



\end{thebibliography}

\end{document}